\documentclass[11pt]{amsart}
\usepackage[T1]{fontenc}

\usepackage[a4paper,margin=1.08in]{geometry}
\usepackage[square,comma,numbers,sort&compress]{natbib}
\usepackage[colorlinks=true,citecolor=red,linkcolor=blue,urlcolor=blue]{hyperref}
\usepackage{amsmath,amssymb,mathtools,mathrsfs}
\usepackage[shortlabels]{enumitem}
\usepackage{aliascnt}

\numberwithin{equation}{section}

\newtheorem{theorem}{Theorem}[section]

\newaliascnt{lemma}{theorem}
\newtheorem{lemma}[lemma]{Lemma}
\aliascntresetthe{lemma}

\newaliascnt{corollary}{theorem}
\newtheorem{corollary}[corollary]{Corollary}
\aliascntresetthe{corollary}

\newaliascnt{proposition}{theorem}

\aliascntresetthe{proposition}

\newaliascnt{example}{theorem}
\newtheorem{example}[example]{Example}
\aliascntresetthe{example}

\newaliascnt{remark}{theorem}
\newtheorem{remark}[remark]{Remark}
\aliascntresetthe{remark}

\usepackage[nameinlink,noabbrev]{cleveref}
\crefname{theorem}{Theorem}{Theorems}
\Crefname{theorem}{Theorem}{Theorems}
\crefname{corollary}{Corollary}{Corollaries}
\Crefname{corollary}{Corollary}{Corollaries}
\crefname{lemma}{Lemma}{Lemmas}
\Crefname{lemma}{Lemma}{Lemmas}
\crefname{example}{Example}{Examples}
\Crefname{example}{Example}{Examples}
\crefname{remark}{Remark}{Remarks}
\Crefname{remark}{Remark}{Remarks}

\providecommand{\PP}{\mathbb P}
\providecommand{\E}{\mathbb E}
\providecommand{\Var}{\operatorname{Var}}

\providecommand{\bft}{\boldsymbol t}
\providecommand{\bfs}{\boldsymbol s}
\providecommand{\bfv}{\boldsymbol v}

\providecommand{\calH}{\mathcal H}

\newcommand{\bfzero}{\boldsymbol{0}}

\newcommand{\dd}{\,\mathrm{d}}

\title[Gaussian fields with a product term in the variance]{Extremes of Gaussian fields with a product term in the variance}
\author{Svyatoslav Novikov}
\address{Department of Actuarial Science, University of Lausanne, UNIL-Dorigny, 1015 Lausanne, Switzerland}
\email{Svyatoslav.Novikov@unil.ch}
\date{\today}

\begin{document}

\begin{abstract}
	We study the high excursion probability of a centered Gaussian field on a square.
	Writing \(\sigma\) and \(r\) for its standard deviation and correlation function,
	we assume that \(\sigma\) has a unique maximum at the corner
	\(\bfzero=(0,0)\) and
	\[
	1-\sigma(\bft) \sim t_1^\beta+t_2^\beta+t_1^a t_2^a ,
	\qquad \bft=(t_1,t_2)\to\bfzero
	\]
	in \(\mathbb R_+^2\).  The local correlation is assumed to satisfy
	\[
	1-r(\bft,\bfs)\sim |t_1-s_1|^\alpha+|t_2-s_2|^\alpha,
	\qquad 0<\alpha<\beta .
	\]
	This product form of the standard-deviation loss is not covered by the usual
	locally additive assumptions.  In the range \(a<\beta/2\), the classical essential rectangle
	at the variance-loss scale no longer captures the leading contribution; the
	relevant localization becomes side-attached and, in one regime, effectively
	one-dimensional. We determine the corresponding high-level asymptotics,
	including the logarithmic and side-dominated regimes which do not arise in the
	locally additive case.
\end{abstract}

\maketitle

\section{Introduction}

Let \(X(\bft),\ \bft=(t_1,t_2)\in[0,T]^2\), be a centered Gaussian field with continuous sample paths, and write \(\bfzero=(0,0)\).  Its standard deviation and correlation function are denoted by
\[
        \sigma(\bft)=\sqrt{\E\{X^2(\bft)\}},
        \qquad
        r(\bft,\bfs)=\frac{\E\{X(\bft)X(\bfs)\}}{\sigma(\bft)\sigma(\bfs)} .
\]
We also write
\[
        \Psi(u)=\PP\{N(0,1)>u\}
\]
for the standard normal survival function.  We are interested in
\[
        p(u)=\PP\left\{\sup_{\bft\in[0,T]^2}X(\bft)>u\right\},
        \qquad u\to\infty,
\]
under the assumption that \(\bfzero\) is the unique point of maximal variance.  For parameters \(0<\alpha\le2\), \(\beta>\alpha\), and \(a>0\), the local model considered in this paper is
\begin{equation}\label{eq:intro-local-model}
        1-r(\bft,\bfs)\sim |t_1-s_1|^\alpha+|t_2-s_2|^\alpha,
        \qquad
        1-\sigma(\bft)\sim t_1^\beta+t_2^\beta+t_1^a t_2^a .
\end{equation}
The restriction \(\alpha<\beta\) is not imposed for conceptual reasons: it keeps us in the regime where the local block constants have the explicit Pickands form described below, and where the transition caused by the product term is most transparent.  The cases \(\alpha=\beta\) and \(\alpha>\beta\) can also be treated, but they introduce additional local constants or single-point regimes which would be less relevant to the mechanism studied here.

The classical double-sum method of Pickands and Piterbarg says, roughly, that the field should be inspected on correlation blocks of side length
\[
        q_u=u^{-2/\alpha},
\]
where \(q_u\) is the mesh on which \(u^2(1-r)\) is of order one.  A correlation block is just a rectangle whose side lengths are proportional to \(q_u\).  On this scale the standardized field, after recentering at the base point of the block, has a non-degenerate fractional-Brownian local limit.  The corresponding block-exceedance probability is \(\Psi(u)\) times a finite Pickands factor, multiplied by the exponential penalty coming from the variance loss at the base point.  Thus the terms \emph{Pickands block} and \emph{Pickands rectangle} below are only shorthand for these \(q_u\)-scale rectangles in the double-sum decomposition; the formal Pickands constant introduced in \Cref{sec:main-results} is the limiting contribution per unit rescaled length.  We mention Piterbarg constants only to indicate what happens in the regimes not treated here: by this we mean the analogous local constants which appear when the deterministic variance drift and the correlation drift live on the same scale.

When the variance loss is additive, say \(t_1^\beta+t_2^\beta\), the relevant variance scale is \(u^{-2/\beta}\).  Since \(\alpha<\beta\), this scale contains many correlation blocks in each coordinate, and the leading term comes from a two-dimensional array of Pickands rectangles.  In that situation one obtains the familiar order
\[
        u^{4/\alpha-4/\beta}\Psi(u),
\]
up to a constant.  This is the standard local picture behind many asymptotic results for non-stationary Gaussian processes and fields; see, for example, \citet{Pickands1969}, \citet{Piterbarg1996}, and the monograph of \citet{LeadbetterLindgrenRootzen1983}.  Related settings in which local additivity is modified or supplemented include the non-additive dependence framework of \citet{BaiDebickiLiu2025} and the vector-valued locally additive fields of \citet{IevlevKriukov2025}.

The product standard-deviation loss in \eqref{eq:intro-local-model}, especially
the range \(a<\beta/2\), is not covered by the usual locally additive theory.
The reason is geometric.  In the additive model, the set on which the
standard-deviation loss is of order \(u^{-2}\) is essentially the rectangle
\[
t_1=O(u^{-2/\beta}),\qquad t_2=O(u^{-2/\beta}),
\]
which is the classical essential rectangle in the Borell--Piterbarg
localization argument.  In the product model the corresponding localization set is
\[
\left\{\boldsymbol{t}\in\mathbb R_+^2:
u^2\bigl(t_1^\beta+t_2^\beta+t_1^a t_2^a\bigr)=O(1)\right\}.
\]
If \(a<\beta/2\), this set is no longer comparable with a single rectangle at
the additive scale.  Indeed, when both coordinates are of order
\(u^{-2/\beta}\), the product term contributes \(u^{2-4a/\beta}\), which tends
to infinity.  Thus the classical essential rectangle does not capture the
leading contribution; the relevant localization is instead attached to the
coordinate sides.

This change of geometry produces the transition studied below.  The threshold is
\[
a_0=\frac{\alpha\beta}{\alpha+\beta}.
\]
For \(a<a_0\), the leading contribution is one-dimensional and comes from the
two sides \([0,T]\times\{0\}\) and \(\{0\}\times[0,T]\).  For
\(a_0\le a<\beta/2\), the leading contribution is still two-dimensional, but it
comes from asymmetric correlation blocks near the sides and contains a
logarithmic factor.  At \(a=\beta/2\), the logarithm disappears and the product
term enters only through a critical limiting integral.  For \(a>\beta/2\), the
product term is negligible on the additive scale and the classical
two-dimensional asymptotic is recovered.  To the best of our knowledge, the
logarithmic and side-dominated regimes are not covered by the existing
Pickands--Piterbarg asymptotic theory for locally additive variance losses.

We also record a simple deterministic trend extension.  For \(\beta=2\), fixed \(c_1,c_2\ge0\), and \(\boldsymbol c=(c_1,c_2)\), we consider
\[
        p_{\boldsymbol c}(u)=\PP\left\{\sup_{\bft\in[0,T]^2}\bigl(X(\bft)-c_1t_1-c_2t_2\bigr)>u\right\}.
\]
This is included mainly to show how the same geometry interacts with a basic trend.  It is in the spirit of threshold-dependent Gaussian-process results such as those of \citet{BaiDebickiHashorvaJi2018}.  The formulae below show the following for \(\beta=2\).  In the side-dominated regime \(0<a<a_0\), where \(a_0=2\alpha/(\alpha+2)\), the fixed linear trend changes the one-dimensional side constant; in the logarithmic product regime \(a_0\le a<1=\beta/2\), it affects only the bounded, non-logarithmic part of the asymptotic and therefore leaves the coefficient of \(\log u\) unchanged; at the critical point \(a=1\), and in the classical regime \(a>1\), it changes the corresponding two-dimensional constants.

The paper is organized as follows.  \Cref{sec:main-results} states the assumptions, the main tail asymptotics (\Cref{thm:main}), the trend corollary (\Cref{cor:trend}), and examples.  \Cref{sec:auxiliary-lemmas} collects the local Pickands estimates, the double-sum bounds, and the integral asymptotics needed in the proof.  \Cref{sec:proofs} proves the main theorem, the trend corollary, and the auxiliary estimates.

\section{Main results}\label{sec:main-results}

For convenience we recall and fix the notation
\[
        E=[0,T]^2,\qquad \bfzero=(0,0),\qquad
        \Psi(u)=\PP\{N(0,1)>u\},
\]
where \(T>0\) is fixed.  We use the following assumptions.

\begin{enumerate}[label=(A\arabic*)]
\item\label{ass:variance}
The field \(X\) is centered, Gaussian, and has continuous sample paths.  Its standard deviation \(\sigma\) satisfies \(\sigma(\bfzero)=1\), \(\sigma(\bft)<1\) for \(\bft\ne\bfzero\), and, as \(\bft\to\bfzero\),
\begin{equation}\label{eq:variance-assumption}
        1-\sigma(\bft)=\bigl(1+o(1)\bigr)V(\bft),
        \qquad
        V(\bft)=t_1^\beta+t_2^\beta+t_1^a t_2^a,
\end{equation}
where \(\beta>0\) and \(a>0\).  Moreover, for every \(\delta>0\),
\[
        \sup_{\bft\in E\setminus[0,\delta]^2}\sigma(\bft)<1 .
\]
\item\label{ass:correlation}
For some \(\alpha\in(0,2]\), uniformly for \(\bft,\bfs\in[0,\delta]^2\) as \(\delta\downarrow0\),
\begin{equation}\label{eq:correlation-assumption}
        1-r(\bft,\bfs)
        =\bigl(1+o(1)\bigr)
        \left(|t_1-s_1|^\alpha+|t_2-s_2|^\alpha\right).
\end{equation}
\item\label{ass:nondegenerate}
For every \(\varepsilon>0\),
\[
        \sup\{r(\bft,\bfs):\bft,\bfs\in E,\\ |\bft-\bfs|\ge\varepsilon\}<1 .
\]
\end{enumerate}

Let \(B_\alpha\) be a standard fractional Brownian motion with Hurst index \(\alpha/2\), that is,
\[
        \operatorname{Var}(B_\alpha(t)-B_\alpha(s))=|t-s|^\alpha .
\]
For \(S>0\), define
\begin{equation}\label{eq:finite-pickands}
        \calH_\alpha(S)
        =\E\left\{\exp\left(\sup_{t\in[0,S]}\bigl(\sqrt2 B_\alpha(t)-t^\alpha\bigr)\right)\right\},
\end{equation}
and let
\begin{equation}\label{eq:pickands-constant}
        \calH_\alpha=\lim_{S\to\infty}\frac{\calH_\alpha(S)}{S}
\end{equation}
be the Pickands constant.  In the estimates below, a rectangle of side lengths \(S_1q_u\) and \(S_2q_u\), with \(q_u=u^{-2/\alpha}\), contributes asymptotically through the finite factor \(\calH_\alpha(S_1)\calH_\alpha(S_2)\).  Letting the rescaled side lengths tend to infinity replaces these finite factors by the intensity \(\calH_\alpha\) per unit rescaled length.  We also write
\begin{equation}\label{eq:constants-main}
        G_\beta=\int_0^\infty e^{-x^\beta} \dd x=\Gamma\left(1+\frac1\beta\right),
        \qquad
        a_0=\frac{\alpha\beta}{\alpha+\beta},
\end{equation}
and, for the critical product case,
\begin{equation}\label{eq:critical-constant}
        K_\beta
        =\int_0^\infty\int_0^\infty
        \exp\left(-x^\beta-y^\beta-x^{\beta/2}y^{\beta/2}\right)\dd x\dd y .
\end{equation}

\begin{theorem}\label{thm:main}
Assume \ref{ass:variance}--\ref{ass:nondegenerate} and \(0<\alpha<\beta\).  Then, as \(u\to\infty\),
\begin{enumerate}[(i)]
\item if \(0<a<a_0\), then
\begin{equation}\label{eq:main-side}
        p(u)
        \sim
        2\calH_\alpha G_\beta\,
        u^{2/\alpha-2/\beta}\Psi(u);
\end{equation}
\item if \(a_0\le a<\beta/2\), then
\begin{equation}\label{eq:main-log}
        p(u)
        \sim
        \calH_\alpha^2\,
        \frac{2(\beta-2a)\Gamma(1/a)}{a^2\beta}\,
        u^{4/\alpha-2/a}\log u\,\Psi(u);
\end{equation}
\item if \(a=\beta/2\), then
\begin{equation}\label{eq:main-critical}
        p(u)
        \sim
        \calH_\alpha^2 K_\beta\,
        u^{4/\alpha-4/\beta}\Psi(u);
\end{equation}
\item if \(a>\beta/2\), then
\begin{equation}\label{eq:main-classical}
        p(u)
        \sim
        \calH_\alpha^2 G_\beta^2\,
        u^{4/\alpha-4/\beta}\Psi(u).
\end{equation}
\end{enumerate}
\end{theorem}

\begin{remark}[How to read the four regimes]
The regimes in \Cref{thm:main} are ordered by the strength of the product term.  In \eqref{eq:main-classical} the product term is negligible at the additive variance scale.  In \eqref{eq:main-critical} it is visible but does not change the order.  In \eqref{eq:main-log} it forces one coordinate to be small and creates the logarithmic integration over scales.  In \eqref{eq:main-side} the two-dimensional block contribution is already smaller than the contribution obtained by summing one-dimensional correlation blocks along the two sides of the square.
\end{remark}

For a trend we keep only the simplest readable version, namely the quadratic side variance \(\beta=2\).  Put
\begin{equation}\label{eq:Lc-def}
        L(c)=\int_0^\infty e^{-x^2-cx}\dd x,
        \qquad c\ge0,
\end{equation}
and
\begin{equation}\label{eq:Kc-def}
        K(c_1,c_2)
        =\int_0^\infty\int_0^\infty
        e^{-x^2-y^2-xy-c_1x-c_2y}\dd x\dd y .
\end{equation}

\begin{corollary}\label{cor:trend}
Assume \ref{ass:variance}--\ref{ass:nondegenerate}, let \(\beta=2\), \(0<\alpha<2\), and fix \(c_1,c_2\ge0\).  Define
\[
        p_{\boldsymbol c}(u)
        =\PP\left\{\sup_{\bft\in E}\bigl(X(\bft)-c_1t_1-c_2t_2\bigr)>u\right\},
        \qquad
        a_0=\frac{2\alpha}{\alpha+2}.
\]
Then, as \(u\to\infty\),
\begin{enumerate}[(i)]
\item if \(0<a<a_0\), then
\begin{equation}\label{eq:trend-side}
        p_{\boldsymbol c}(u)
        \sim
        \calH_\alpha\bigl(L(c_1)+L(c_2)\bigr)
        u^{2/\alpha-1}\Psi(u);
\end{equation}
\item if \(a_0\le a<1\), then
\begin{equation}\label{eq:trend-log}
        p_{\boldsymbol c}(u)
        \sim
        \calH_\alpha^2\,
        \frac{2(1-a)\Gamma(1/a)}{a^2}\,
        u^{4/\alpha-2/a}\log u\,\Psi(u);
\end{equation}
\item if \(a=1\), then
\begin{equation}\label{eq:trend-critical}
        p_{\boldsymbol c}(u)
        \sim
        \calH_\alpha^2 K(c_1,c_2)
        u^{4/\alpha-2}\Psi(u);
\end{equation}
\item if \(a>1\), then
\begin{equation}\label{eq:trend-classical}
        p_{\boldsymbol c}(u)
        \sim
        \calH_\alpha^2 L(c_1)L(c_2)
        u^{4/\alpha-2}\Psi(u).
\end{equation}
\end{enumerate}
\end{corollary}
\begin{remark}[Why the trend corollary is stated for \(\beta=2\)]
	The restriction \(\beta=2\) in \Cref{cor:trend} is mainly a matter of balance.  A fixed linear trend contributes to the exponential term on the scale \(t\asymp u^{-1}\), whereas the variance loss \(t^\beta\) is felt on the scale \(t\asymp u^{-2/\beta}\).  Hence, if \(\beta<2\), then \(u^{-2/\beta}=o(u^{-1})\), and the trend is negligible on the variance scale; the corresponding fixed-trend analogue of \Cref{thm:main} has the same leading asymptotics.  If \(\beta=2\), the two effects live on the same scale, which is exactly the situation covered in \Cref{cor:trend}: the powers of \(u\) are unchanged, but the constants may change.  If \(\beta>2\), then \(u^{-1}=o(u^{-2/\beta})\), so a positive linear trend can cut off the variance scale before the variance loss becomes effective, and the order of the asymptotics may change.  We do not treat the cases \(\beta\ne2\) here, since for \(\beta<2\) the leading asymptotics are unchanged, while the extension \(\beta>2\) is not conceptually difficult but would require separate case distinctions and would obscure the product-term transition which is the focus of the paper.
\end{remark}

\begin{example}[The quadratic model with Brownian local structure]\label{ex:original}
Let \(\alpha=1\), \(\beta=2\), and \(c_1=c_2=0\).  Since \(\calH_1=1\), \(G_2=\Gamma(3/2)=\sqrt\pi/2\), and \(a_0=2/3\), \Cref{thm:main} gives
\[
        p(u)\sim
        \begin{cases}
        \sqrt\pi\,u\Psi(u),
        &0<a<2/3,\\[0.4em]
        \dfrac{3\sqrt\pi}{4}\,u\log u\,\Psi(u),
        &a=2/3,\\[0.8em]
        \dfrac{2(1-a)\Gamma(1/a)}{a^2}\,
        u^{4-2/a}\log u\,\Psi(u),
        &2/3<a<1,\\[0.8em]
        \dfrac{\pi}{3\sqrt3}\,u^2\Psi(u),
        &a=1,\\[0.8em]
        \dfrac{\pi}{4}\,u^2\Psi(u),
        &a>1.
        \end{cases}
\]
The first line is the side-dominated regime.  The second and third lines correspond to the logarithmic regime generated by asymmetric two-dimensional correlation blocks.  At \(a=1\),
\[
        \int_0^\infty\int_0^\infty e^{-x^2-y^2-xy}\dd x\dd y
        =\frac{\pi}{3\sqrt3}.
\]
\end{example}

\begin{example}[The critical product case]\label{ex:critical}
For any \(0<\alpha<\beta\), the variance loss
\[
        t_1^\beta+t_2^\beta+t_1^{\beta/2}t_2^{\beta/2}
\]
falls under \Cref{thm:main}(iii).  The order is the same as in the additive model,
\[
        u^{4/\alpha-4/\beta}\Psi(u),
\]
but the additive constant \(G_\beta^2\) is replaced by the critical integral \(K_\beta\).  For \(\beta=2\), this gives the explicit constant \(\pi/(3\sqrt3)\).
\end{example}

\begin{example}[A linear trend on the sides]\label{ex:trend}
Let \(\alpha=1\), \(\beta=2\), and \(0<a<2/3\).  Then the leading contribution is produced by the two sides, and Corollary~\ref{cor:trend} gives
\[
        p_{\boldsymbol c}(u)
        \sim
        \bigl(L(c_1)+L(c_2)\bigr)u\Psi(u),
\]
where \(L\) is defined in \eqref{eq:Lc-def}.  Thus a positive drift on the first side changes only the constant associated with \([0,T]\times\{0\}\), and similarly for the second side.  In contrast, when \(2/3\le a<1\), the leading logarithmic constant is unchanged by any fixed non-negative linear trend.
\end{example}

\section{Auxiliary lemmas}\label{sec:auxiliary-lemmas}

We collect the estimates used in the proof.  Recall that
\begin{equation}\label{eq:qu}
        q_u=u^{-2/\alpha}.
\end{equation}
For \(S_1,S_2>0\) and \(\bfv\in[0,T]^2\), write
\[
        B_u(\bfv;S_1,S_2)
        =[v_1,v_1+S_1q_u]\times[v_2,v_2+S_2q_u].
\]
When \(S_1=S_2=S\), write simply \(B_u(\bfv;S)\).  The following local estimate is the formal version of the block description in the introduction.

\begin{lemma}\label{lem:local-pickands}
Assume \ref{ass:variance}--\ref{ass:correlation}.  Fix \(S_1,S_2>0\).  For every \(\varepsilon>0\), there exists \(\delta>0\) such that, uniformly for \(\bfv\in[0,\delta]^2\) and all sufficiently large \(u\),
\begin{align}\label{eq:block-upper}
&\PP\left\{\sup_{\bft\in B_u(\bfv;S_1,S_2)}X(\bft)>u\right\}  
\notag\\
&\quad\le
(1+\varepsilon)\calH_\alpha(S_1)\calH_\alpha(S_2)\Psi(u)
\exp\left(-(1-\varepsilon)u^2V(\bfv)\right).
\end{align}
Moreover, if \(R_u\to\infty\) and the estimate is restricted to blocks satisfying \(v_i\ge R_uq_u\), \(i=1,2\), then
\begin{align}\label{eq:block-lower}
&\PP\left\{\sup_{\bft\in B_u(\bfv;S_1,S_2)}X(\bft)>u\right\}  
\notag\\
&\quad\ge
(1-\varepsilon)\calH_\alpha(S_1)\calH_\alpha(S_2)\Psi(u)
\exp\left(-(1+\varepsilon)u^2V(\bfv)\right)
\end{align}
for all sufficiently large \(u\), uniformly over such blocks contained in \([0,\delta]^2\).
\end{lemma}

\begin{lemma}\label{lem:rectangle-bounds}
Assume \ref{ass:variance}--\ref{ass:nondegenerate}, and let \(M>0\).  For every \(\varepsilon>0\), there exists \(\delta>0\) such that, as \(u\to\infty\),
\begin{align}\label{eq:upper-rectangles}
&\PP\left\{\sup_{\bft\in[0,\delta]^2}X(\bft)>u\right\}
\notag\\
&\quad\le
(1+\varepsilon)\Psi(u)
\biggl[
\calH_\alpha(M)^2
+2\frac{\calH_\alpha(M)^2}{M}u^{2/\alpha}
       \int_0^\delta e^{-(1-\varepsilon)u^2x^\beta}\dd x
\notag\\
&\hspace{7.5em}
+\frac{\calH_\alpha(M)^2}{M^2}u^{4/\alpha}
       \int_0^\delta\int_0^\delta e^{-(1-\varepsilon)u^2V(x,y)}\dd x\dd y
\biggr].
\end{align}
Furthermore, define
\[
        I_\delta(u)=\int_0^\delta\int_0^\delta e^{-u^2V(x,y)}\dd x\dd y .
\]
If \(R_u\to\infty\) and \(R_uq_u=o(u^{-2/\beta})\), then for some function \(\eta(M)\downarrow0\) as \(M\to\infty\),
\begin{align}\label{eq:lower-rectangles}
&\PP\left\{\sup_{\bft\in[0,\delta]^2}X(\bft)>u\right\}
\notag\\
&\quad\ge
(1-\varepsilon)\frac{\calH_\alpha(M)^2}{M^2}u^{4/\alpha}
\Psi(u)
\int_{R_uq_u}^\delta\int_{R_uq_u}^\delta
      e^{-(1+\varepsilon)u^2V(x,y)}\dd x\dd y
\notag\\
&\hspace{2.5em}
-\eta(M)\Psi(u)u^{4/\alpha}
\int_0^\delta\int_0^\delta e^{-(1-\varepsilon)u^2V(x,y)}\dd x\dd y
+o\left(\Psi(u)u^{4/\alpha}I_\delta(u)\right).
\end{align}
\end{lemma}

For the side-dominated regime we need a sharper upper bound near the sides.  Let \(\lambda>0\) and define the side strips
\[
        S_{1,u}(\lambda)=[0,\delta]\times[0,\lambda q_u],
        \qquad
        S_{2,u}(\lambda)=[0,\lambda q_u]\times[0,\delta].
\]

\begin{lemma}\label{lem:side-bounds}
Assume \ref{ass:variance}--\ref{ass:nondegenerate}.  Fix \(M>0\), \(\lambda>0\), and \(\varepsilon>0\).  There exists \(\delta>0\) such that, as \(u\to\infty\),
\begin{align}\label{eq:upper-sides}
&\PP\left\{\sup_{\bft\in S_{1,u}(\lambda)\cup S_{2,u}(\lambda)}X(\bft)>u\right\}
\notag\\
&\quad\le
(1+\varepsilon)\Psi(u)
\biggl[
2\calH_\alpha(\lambda)^2
+\frac{\calH_\alpha(M)\calH_\alpha(\lambda)}{M}u^{2/\alpha}
      \sum_{i=1}^2\int_0^\delta e^{-(1-\varepsilon)u^2x^\beta}\dd x
\biggr].
\end{align}
In addition, if \(R_u\to\infty\) and \(R_uq_u=o(u^{-2/\beta})\), then for some function \(\eta_1(M)\downarrow0\) as \(M\to\infty\),
\begin{align}\label{eq:lower-sides}
&\PP\left\{\sup_{t\in[0,\delta]}X(t,0)>u\ \text{or}\ \sup_{t\in[0,\delta]}X(0,t)>u\right\}
\notag\\
&\quad\ge
(1-\varepsilon)\frac{\calH_\alpha(M)}{M}u^{2/\alpha}
\Psi(u)
\sum_{i=1}^2\int_{R_uq_u}^\delta e^{-(1+\varepsilon)u^2x^\beta}\dd x
\notag\\
&\hspace{2.5em}
-\eta_1(M)\Psi(u)u^{2/\alpha}\int_0^\delta e^{-(1-\varepsilon)u^2x^\beta}\dd x+o\left(\Psi(u)u^{2/\alpha-2/\beta}\right).
\end{align}
\end{lemma}

The next lemma contains the integral asymptotics responsible for the transition.

\begin{lemma}\label{lem:integrals}
Let \(\delta>0\), \(\gamma>0\), and
\[
        I_\gamma(u)=\int_0^\delta\int_0^\delta
        e^{-\gamma u^2(x^\beta+y^\beta+x^ay^a)}\dd x\dd y .
\]
Then, as \(u\to\infty\),
\begin{enumerate}[(i)]
\item if \(0<a<\beta/2\), then
\begin{equation}\label{eq:integral-log}
        I_\gamma(u)
        \sim
        \frac{2(\beta-2a)\Gamma(1/a)}{a^2\beta\,\gamma^{1/a}}
        u^{-2/a}\log u;
\end{equation}
\item if \(a=\beta/2\), then
\begin{equation}\label{eq:integral-critical}
        I_\gamma(u)
        \sim
        u^{-4/\beta}\int_0^\infty\int_0^\infty
        e^{-\gamma(x^\beta+y^\beta+x^{\beta/2}y^{\beta/2})}\dd x\dd y;
\end{equation}
\item if \(a>\beta/2\), then
\begin{equation}\label{eq:integral-classical}
        I_\gamma(u)
        \sim
        \gamma^{-2/\beta}G_\beta^2u^{-4/\beta}.
\end{equation}
\end{enumerate}
Moreover, if \(0<\alpha<\beta\), \(a\ge a_0=\alpha\beta/(\alpha+\beta)\), \(R_u\to\infty\), and \(\log(R_u)=o(\log(u))\), then in the corresponding cases above
\begin{equation}\label{eq:integral-with-cutoff}
        \int_{R_uq_u}^\delta\int_{R_uq_u}^\delta
        e^{-\gamma u^2(x^\beta+y^\beta+x^ay^a)}\dd x\dd y
        \sim I_\gamma(u).
\end{equation}
Finally,
\begin{equation}\label{eq:side-integral}
        \int_0^\delta e^{-\gamma u^2x^\beta}\dd x
        \sim
        \gamma^{-1/\beta}G_\beta u^{-2/\beta}.
\end{equation}
\end{lemma}

For the trend corollary, only the following elementary modifications are needed.  The functions \(L\) and \(K\) are those defined in \eqref{eq:Lc-def} and \eqref{eq:Kc-def}.

\begin{lemma}\label{lem:trend-integrals}
Let \(c_1,c_2\ge0\), \(\delta>0\), and define
\[
        I_{\boldsymbol c}(u)
        =\int_0^\delta\int_0^\delta
        e^{-u^2(x^2+y^2+x^ay^a)-u(c_1x+c_2y)}\dd x\dd y .
\]
Then, as \(u\to\infty\),
\begin{align}
        I_{\boldsymbol c}(u)
        &\sim
        \frac{2(1-a)\Gamma(1/a)}{a^2}u^{-2/a}\log u,
        &&0<a<1, \label{eq:trend-integral-log}\\
        I_{\boldsymbol c}(u)
        &\sim
        K(c_1,c_2)u^{-2},
        &&a=1, \label{eq:trend-integral-critical}\\
        I_{\boldsymbol c}(u)
        &\sim
        L(c_1)L(c_2)u^{-2},
        &&a>1. \label{eq:trend-integral-classical}
\end{align}
Also,
\begin{equation}\label{eq:trend-side-integral}
        \int_0^\delta e^{-u^2x^2-ucx}\dd x
        \sim u^{-1}L(c),
        \qquad c\ge0 .
\end{equation}
\end{lemma}

\section{Proofs}\label{sec:proofs}

\subsection{Proofs of the local probability bounds}

\begin{proof}[Proof of \Cref{lem:local-pickands}]
	Write
	\[
	\overline X(\bft)=\frac{X(\bft)}{\sigma(\bft)},
	\qquad
	D_u(\bfv)=B_u(\bfv;S_1,S_2).
	\]
	We first recall the finite-block Pickands estimate in the form in which it will be used.
	For every fixed \(S_1,S_2>0\) and every \(\eta>0\), after decreasing \(\delta>0\) if
	necessary, uniformly for \(\bfv\in[0,\delta]^2\), for all large \(u\), and for all
	thresholds \(w=u(1+\theta)\) with \(0\le \theta\le C_\delta\), where
	\(C_\delta\downarrow0\) as \(\delta\downarrow0\),
	\[
	\begin{aligned}
		\PP\left\{\sup_{\bft\in D_u(\bfv)}\overline X(\bft)>w\right\}
		&\le
		(1+\eta)\calH_\alpha(S_1)\calH_\alpha(S_2)\Psi(w),                                      \\
		\PP\left\{\sup_{\bft\in D_u(\bfv)}\overline X(\bft)>w\right\}
		&\ge
		(1-\eta)\calH_\alpha(S_1)\calH_\alpha(S_2)\Psi(w).
	\end{aligned}
	\]
	Indeed, applying the usual finite-block Pickands lemma to the standardized field
	\(\overline X\), see for instance \cite[Lemma 6.1]{Piterbarg1996}, and using
	\ref{ass:correlation}, the local tangent field on a rectangle is
	\[
	\sqrt2 B_{\alpha,1}(s_1)-s_1^\alpha
	+\sqrt2 B_{\alpha,2}(s_2)-s_2^\alpha,
	\]
	where \(B_{\alpha,1}\) and \(B_{\alpha,2}\) are independent fractional Brownian motions.
	Hence the two-dimensional Pickands constant factors into
	\(\calH_\alpha(S_1)\calH_\alpha(S_2)\).  If \(w/u\ne1\), then the block sides in the
	\(w^{-2/\alpha}\)-scale are \(S_i(w/u)^{2/\alpha}\); since \(w/u\) is uniformly close to
	one and \(S\mapsto\calH_\alpha(S)\) is continuous, this only changes the above estimates by
	the factor \(1\pm\eta\).
	
	We prove the upper bound first.  By \ref{ass:variance} and the monotonicity of \(V\) on
	\(\mathbb R_+^2\), for all sufficiently small \(\delta\) and all large \(u\),
	\[
	\sigma(\bft)
	\le 1-(1-\eta)V(\bft)
	\le 1-(1-\eta)V(\bfv),
	\qquad \bft\in D_u(\bfv).
	\]
	Reducing \(\delta\) once more, this implies
	\[
	\frac{u}{\sigma(\bft)}
	\ge u\bigl(1+(1-2\eta)V(\bfv)\bigr),
	\qquad \bft\in D_u(\bfv).
	\]
	Therefore
	\[
	\begin{aligned}
		\PP\left\{\sup_{\bft\in D_u(\bfv)}X(\bft)>u\right\}
		&\le
		\PP\left\{\sup_{\bft\in D_u(\bfv)}\overline X(\bft)
		>u\bigl(1+(1-2\eta)V(\bfv)\bigr)\right\}       \\
		&\le
		(1+\eta)\calH_\alpha(S_1)\calH_\alpha(S_2)
		\Psi\left(u\bigl(1+(1-2\eta)V(\bfv)\bigr)\right).
	\end{aligned}
	\]
	Since \(V(\bfv)\to0\) uniformly for \(\bfv\in[0,\delta]^2\) as \(\delta\downarrow0\),
	Mills' ratio gives, uniformly in the present range of \(\bfv\),
	\[
	\Psi\left(u\bigl(1+(1-2\eta)V(\bfv)\bigr)\right)
	\le
	(1+\eta)\Psi(u)\exp\left(-(1-3\eta)u^2V(\bfv)\right),
	\]
	after decreasing \(\delta\) and taking \(u\) large.  Choosing \(\eta\) small enough in terms
	of the prescribed \(\varepsilon\) gives \eqref{eq:block-upper}.
	
	We now prove the lower bound.  Assume that \(v_i\ge R_uq_u\), \(i=1,2\), with
	\(R_u\to\infty\).  Since \(S_1,S_2\) are fixed, uniformly for
	\(\bft\in D_u(\bfv)\),
	\[
	v_i\le t_i\le v_i\left(1+\frac{S_i}{R_u}\right),
	\qquad i=1,2,
	\]
	for all large \(u\).  Hence each of the three monomials
	\(t_1^\beta\), \(t_2^\beta\), and \(t_1^at_2^a\) differs from its value at \(\bfv\) by a
	relative \(O(R_u^{-1})\) amount.  Consequently
	\[
	\sup_{\bft\in D_u(\bfv)}|V(\bft)-V(\bfv)|
	=O(R_u^{-1})V(\bfv)
	=o(V(\bfv)),
	\]
	uniformly over the blocks under consideration.  Using \ref{ass:variance}, it follows that
	\[
	\sigma(\bft)\ge 1-(1+\eta)V(\bft)
	\ge 1-(1+2\eta)V(\bfv),
	\qquad \bft\in D_u(\bfv),
	\]
	for all large \(u\).  After reducing \(\delta\), we also have
	\[
	\frac{u}{\sigma(\bft)}
	\le u\bigl(1+(1+3\eta)V(\bfv)\bigr),
	\qquad \bft\in D_u(\bfv).
	\]
	Thus
	\[
	\begin{aligned}
		\PP\left\{\sup_{\bft\in D_u(\bfv)}X(\bft)>u\right\}
		&\ge
		\PP\left\{\sup_{\bft\in D_u(\bfv)}\overline X(\bft)
		>u\bigl(1+(1+3\eta)V(\bfv)\bigr)\right\}       \\
		&\ge
		(1-\eta)\calH_\alpha(S_1)\calH_\alpha(S_2)
		\Psi\left(u\bigl(1+(1+3\eta)V(\bfv)\bigr)\right).
	\end{aligned}
	\]
	The lower Mills-ratio estimate, again uniformly in the localized region, gives
	\[
	\Psi\left(u\bigl(1+(1+3\eta)V(\bfv)\bigr)\right)
	\ge
	(1-\eta)\Psi(u)\exp\left(-(1+4\eta)u^2V(\bfv)\right).
	\]
	Taking \(\eta\) sufficiently small and relabelling constants yields \eqref{eq:block-lower}.
\end{proof}

\begin{proof}[Proof of \Cref{lem:rectangle-bounds}]
	Fix \(\varepsilon>0\), and let \(0<\varepsilon_0<\varepsilon/10\).  Choose
	\(\delta>0\) so small that \Cref{lem:local-pickands} is valid, with
	\(\varepsilon_0\) in place of \(\varepsilon\), for all blocks contained in
	\([0,2\delta]^2\).  Put \(q_u=u^{-2/\alpha}\).
	
	We start with the upper bound.  Let
	\[
	N_u=\left\lceil \frac{\delta}{M q_u}\right\rceil
	\]
	and cover \([0,\delta]^2\) by the rectangles
	\[
	P_{\boldsymbol{k}}
	=
	[k_1Mq_u,(k_1+1)Mq_u]\times
	[k_2Mq_u,(k_2+1)Mq_u],
	\qquad
	0\le k_1,k_2\le N_u-1 .
	\]
	For all large \(u\), these rectangles are contained in \([0,2\delta]^2\).
	The union bound and \Cref{lem:local-pickands} give
	\[
	\begin{aligned}
		&\PP\left\{\sup_{\bft\in[0,\delta]^2}X(\bft)>u\right\}
		\\
		&\qquad\le
		(1+\varepsilon_0)\calH_\alpha(M)^2\Psi(u)
		\sum_{0\le k_1,k_2\le N_u-1}
		\exp\left\{-(1-\varepsilon_0)u^2
		V(k_1Mq_u,k_2Mq_u)\right\}.
	\end{aligned}
	\]
	We estimate the last sum by monotonicity of \(V\).  The term
	\((k_1,k_2)=(0,0)\) gives the corner contribution.  If, for example,
	\(k_1=0\) and \(k_2\ge1\), then
	\(V(0,k_2Mq_u)=(k_2Mq_u)^\beta\), and
	\[
	\exp\left\{-(1-\varepsilon_0)u^2(k_2Mq_u)^\beta\right\}
	\le
	\frac{u^{2/\alpha}}{M}
	\int_{(k_2-1)Mq_u}^{k_2Mq_u}
	\exp\left\{-(1-2\varepsilon_0)u^2x^\beta\right\}\dd x .
	\]
	The same estimate applies to the other side.  Finally, if \(k_1,k_2\ge1\), then
	for \(x\in[(k_1-1)Mq_u,k_1Mq_u]\) and
	\(y\in[(k_2-1)Mq_u,k_2Mq_u]\),
	\[
	V(x,y)\le V(k_1Mq_u,k_2Mq_u),
	\]
	and therefore
	\[
	\begin{aligned}
		&\exp\left\{-(1-\varepsilon_0)u^2
		V(k_1Mq_u,k_2Mq_u)\right\}                                      \\
		&\qquad\le
		\frac{u^{4/\alpha}}{M^2}
		\int_{(k_1-1)Mq_u}^{k_1Mq_u}
		\int_{(k_2-1)Mq_u}^{k_2Mq_u}
		\exp\left\{-(1-2\varepsilon_0)u^2V(x,y)\right\}\dd x\dd y .
	\end{aligned}
	\]
	Summing these estimates and then replacing \(\varepsilon_0\) by a comparable
	multiple of \(\varepsilon\) gives \eqref{eq:upper-rectangles}.
	
	We now prove the lower bound.  It is convenient to use separated rectangles, in
	order to avoid estimating intersections of adjacent blocks.  Let
	\[
	m=m(M)=\lfloor M^{1/2}\rfloor\vee1,
	\qquad
	L=M+m,
	\qquad
	r_u=R_uq_u .
	\]
	Thus \(m\to\infty\) and \(m/M\to0\) as \(M\to\infty\).  For all large \(u\),
	\(r_u\to0\).  Let
	\[
	n_u=
	\left\lfloor
	\frac{\delta-r_u-Mq_u}{Lq_u}
	\right\rfloor
	\]
	and, for \(0\le k_1,k_2\le n_u\), define
	\[
	\bfv_{\boldsymbol{k}}
	=
	(r_u+k_1Lq_u,\ r_u+k_2Lq_u),
	\qquad
	P_{\boldsymbol{k}}=B_u(\bfv_{\boldsymbol{k}};M,M).
	\]
	The blocks \(P_{\boldsymbol{k}}\) are contained in \([r_u,\delta]^2\), and
	their mutual gaps have width at least \(m q_u\).  Set
	\[
	A_{\boldsymbol{k}}
	=
	\left\{\sup_{\bft\in P_{\boldsymbol{k}}}X(\bft)>u\right\}.
	\]
	By Bonferroni's inequality,
	\[
	\PP\left\{\sup_{\bft\in[0,\delta]^2}X(\bft)>u\right\}
	\ge
	\sum_{\boldsymbol{k}}\PP(A_{\boldsymbol{k}})
	-
	\sum_{\boldsymbol{k}\ne\boldsymbol{l}}
	\PP(A_{\boldsymbol{k}}\cap A_{\boldsymbol{l}}).
	\]
	
	Since \(v_{\boldsymbol{k},i}\ge r_u=R_uq_u\), \(i=1,2\), the lower estimate in
	\Cref{lem:local-pickands} yields
	\begin{align}
	\label{explb}
	\sum_{\boldsymbol{k}}\PP(A_{\boldsymbol{k}})
	\ge
	(1-\varepsilon_0)\calH_\alpha(M)^2\Psi(u)
	\sum_{\boldsymbol{k}}
	\exp\left\{-(1+\varepsilon_0)u^2V(\bfv_{\boldsymbol{k}})\right\}.
	\end{align}
	Let
	\[
	\mathcal U_u=\bigcup_{\boldsymbol{k}}P_{\boldsymbol{k}}.
	\]
	Since \(V\) is coordinatewise increasing on \(\mathbb R_+^2\),
	\begin{align}
		\label{intlb}
		\sum_{\boldsymbol{k}}
		\exp\left\{-(1+\varepsilon_0)u^2V(\bfv_{\boldsymbol{k}})\right\}
		\ge
		\frac{u^{4/\alpha}}{M^2}
		\int_{\mathcal U_u}
		\exp\left\{-(1+\varepsilon)u^2V(x,y)\right\}\dd x\dd y .
	\end{align}
	The gaps occupy only a proportion \(1-\frac{M^2}{(M+m)^2}= O(m/M)\).  Hence,
	because the integrands
	\[
	\exp\left\{-(1+\varepsilon)u^2V(x,y)\right\},
	\qquad
	\exp\left\{-(1-\varepsilon)u^2V(x,y)\right\},
	\]
	are coordinatewise decreasing, a deterministic comparison of the retained
	blocks with the omitted gaps gives
	\[
	\begin{aligned}
		&\int_{r_u}^{\delta}\int_{r_u}^{\delta}
		\exp\left\{-(1+\varepsilon)u^2V(x,y)\right\}\dd x\dd y
		-
		\int_{\mathcal U_u}
		\exp\left\{-(1+\varepsilon)u^2V(x,y)\right\}\dd x\dd y            \\
		&\qquad\le
		C\frac{m}{M}
		\int_0^\delta\int_0^\delta
		\exp\left\{-(1-\varepsilon)u^2V(x,y)\right\}\dd x\dd y
		+
		o\bigl(I_\delta(u)\bigr),
	\end{aligned}
	\]
	where the \(o(I_\delta(u))\) term comes only from the last incomplete strips
	near \(x=\delta\) or \(y=\delta\), whose contribution is exponentially small.
	Since \(\calH_\alpha(M)^2/M^2\) is bounded as \(M\to\infty\), by \eqref{explb} and \eqref{intlb} the single sum is
	therefore bounded from below by
	\[
	\begin{aligned}
		&(1-\varepsilon)\frac{\calH_\alpha(M)^2}{M^2}
		u^{4/\alpha}\Psi(u)
		\int_{r_u}^{\delta}\int_{r_u}^{\delta}
		\exp\left\{-(1+\varepsilon)u^2V(x,y)\right\}\dd x\dd y             \\
		&\qquad
		-\eta_1(M)\Psi(u)u^{4/\alpha}
		\int_0^\delta\int_0^\delta
		\exp\left\{-(1-\varepsilon)u^2V(x,y)\right\}\dd x\dd y
		+
		o\left(\Psi(u)u^{4/\alpha}I_\delta(u)\right),
	\end{aligned}
	\]
	with \(\eta_1(M)\downarrow0\).
	
	We now estimate the double sum.  Put
	\[
	\overline X(\bft)=\frac{X(\bft)}{\sigma(\bft)} .
	\]
	By \ref{ass:variance}, after decreasing \(\delta\) if necessary, for every block
	\(P_{\boldsymbol{k}}\) used in the lower bound and every \(\bft\in P_{\boldsymbol{k}}\),
	\[
	\frac{u}{\sigma(\bft)}
	\ge
	w_{\boldsymbol{k}},
	\qquad
	w_{\boldsymbol{k}}
	:=
	u\bigl(1+(1-2\varepsilon_0)V(\bfv_{\boldsymbol{k}})\bigr).
	\]
	Hence
	\[
	A_{\boldsymbol{k}}
	\subset
	\left\{
	\sup_{\bft\in P_{\boldsymbol{k}}}\overline X(\bft)
	>w_{\boldsymbol{k}}
	\right\}.
	\]
	
	Let \(\boldsymbol{k}\ne\boldsymbol{l}\), and define
	\[
	\rho_{\boldsymbol{k},\boldsymbol{l}}
	=
	\sup_{\bft\in P_{\boldsymbol{k}},\,\bfs\in P_{\boldsymbol{l}}}
	r(\bft,\bfs),
	\qquad
	w_{\boldsymbol{k},\boldsymbol{l}}
	=
	w_{\boldsymbol{k}}\wedge w_{\boldsymbol{l}} .
	\]
	For \((\bft,\bfs)\in P_{\boldsymbol{k}}\times P_{\boldsymbol{l}}\), introduce the
	four-dimensional centered Gaussian field
	\[
	Y_{\boldsymbol{k},\boldsymbol{l}}(\bft,\bfs)
	=
	\frac{\overline X(\bft)+\overline X(\bfs)}
	{\sqrt{2+2\rho_{\boldsymbol{k},\boldsymbol{l}}}} .
	\]
	Since \(r(\bft,\bfs)\le \rho_{\boldsymbol{k},\boldsymbol{l}}\), we have
	\[
	\Var\bigl(Y_{\boldsymbol{k},\boldsymbol{l}}(\bft,\bfs)\bigr)
	=
	\frac{2+2r(\bft,\bfs)}
	{2+2\rho_{\boldsymbol{k},\boldsymbol{l}}}
	\le 1 .
	\]
	Moreover, by \ref{ass:correlation}, uniformly over the localized region,
	\[
	\begin{aligned}
		&\E\left[
		Y_{\boldsymbol{k},\boldsymbol{l}}(\bft,\bfs)
		-
		Y_{\boldsymbol{k},\boldsymbol{l}}(\bft',\bfs')
		\right]^2                                                \\
		&\qquad\le
		C\left(
		|\bft-\bft'|^\alpha+|\bfs-\bfs'|^\alpha
		\right),
	\end{aligned}
	\]
	where \(C\) does not depend on \(u,M,\boldsymbol{k},\boldsymbol{l}\).
	
	If both
	\(\overline X(\bft)>w_{\boldsymbol{k}}\) and
	\(\overline X(\bfs)>w_{\boldsymbol{l}}\), then
	\[
	Y_{\boldsymbol{k},\boldsymbol{l}}(\bft,\bfs)
	>
	z_{\boldsymbol{k},\boldsymbol{l}},
	\qquad
	z_{\boldsymbol{k},\boldsymbol{l}}
	:=
	w_{\boldsymbol{k},\boldsymbol{l}}
	\sqrt{\frac{2}{1+\rho_{\boldsymbol{k},\boldsymbol{l}}}} .
	\]
	Therefore
	\[
	A_{\boldsymbol{k}}\cap A_{\boldsymbol{l}}
	\subset
	\left\{
	\sup_{(\bft,\bfs)\in P_{\boldsymbol{k}}\times P_{\boldsymbol{l}}}
	Y_{\boldsymbol{k},\boldsymbol{l}}(\bft,\bfs)
	>
	z_{\boldsymbol{k},\boldsymbol{l}}
	\right\}.
	\]
	
	Applying Piterbarg inequality \cite[Lemma 5.1]{gammareflected} with \(n=4\) to
	\(D=P_{\boldsymbol{k}}\times P_{\boldsymbol{l}}\), translated to a rectangle
	with side lengths \(Mq_u,Mq_u,Mq_u,Mq_u\), gives
	\[
	\begin{aligned}
		\PP(A_{\boldsymbol{k}}\cap A_{\boldsymbol{l}})
		&\le
		C\left(1+Mq_u z_{\boldsymbol{k},\boldsymbol{l}}^{2/\alpha}\right)^4
		\Psi(z_{\boldsymbol{k},\boldsymbol{l}}).
	\end{aligned}
	\]
	Since \(q_u=u^{-2/\alpha}\) and
	\(z_{\boldsymbol{k},\boldsymbol{l}}\le Cu\) in the localized region, we have
	\[
	1+Mq_u z_{\boldsymbol{k},\boldsymbol{l}}^{2/\alpha}
	\le C(1+M),
	\]
	and hence
	\[
	\PP(A_{\boldsymbol{k}}\cap A_{\boldsymbol{l}})
	\le
	C(1+M)^4\Psi(z_{\boldsymbol{k},\boldsymbol{l}}).
	\]
	
	It remains to estimate the level \(z_{\boldsymbol{k},\boldsymbol{l}}\).  Denote
	\[
	d_{\boldsymbol{k},\boldsymbol{l}}
	=
	\max_{i=1,2}\bigl(L|k_i-l_i|-M\bigr)_+ .
	\]
	The distance between the two rectangles is at least
	\(d_{\boldsymbol{k},\boldsymbol{l}}q_u\) in one coordinate.  Hence, by
	\ref{ass:correlation} when this distance tends to zero, and by
	\ref{ass:nondegenerate} when the rectangles are separated by a fixed positive
	distance, there exists \(c>0\), independent of
	\(u,M,\boldsymbol{k},\boldsymbol{l}\), such that
	\[
	1-\rho_{\boldsymbol{k},\boldsymbol{l}}
	\ge
	c\,d_{\boldsymbol{k},\boldsymbol{l}}^\alpha u^{-2}.
	\]
	Indeed, if \(d_{\boldsymbol{k},\boldsymbol{l}}q_u\) is small, this follows from
	the local correlation expansion, since
	\[
	(d_{\boldsymbol{k},\boldsymbol{l}}q_u)^\alpha
	=
	d_{\boldsymbol{k},\boldsymbol{l}}^\alpha u^{-2}.
	\]
	If \(d_{\boldsymbol{k},\boldsymbol{l}}q_u\) is bounded away from zero, then
	\ref{ass:nondegenerate} gives \(1-\rho_{\boldsymbol{k},\boldsymbol{l}}\ge c_0\),
	whereas
	\(d_{\boldsymbol{k},\boldsymbol{l}}^\alpha u^{-2}\) is uniformly bounded; reducing \(c\) gives the same bound.
	
	Consequently,
	\[
	\frac{2}{1+\rho_{\boldsymbol{k},\boldsymbol{l}}}
	\ge
	1+c\,d_{\boldsymbol{k},\boldsymbol{l}}^\alpha u^{-2},
	\]
	with \(c>0\) possibly changed.  Since
	\[
	w_{\boldsymbol{k},\boldsymbol{l}}^2
	=u^2(1+(1-2 \varepsilon_0)\min\{V(\bfv_{\boldsymbol{k}}),V(\bfv_{\boldsymbol{l}})\})^2
	\ge
	u^2\left(
	1+2(1-3\varepsilon_0)
	\min\{V(\bfv_{\boldsymbol{k}}),V(\bfv_{\boldsymbol{l}})\}
	\right),
	\]
	we obtain
	\[
	z_{\boldsymbol{k},\boldsymbol{l}}^2
	\ge
	u^2
	+c\,d_{\boldsymbol{k},\boldsymbol{l}}^\alpha
	+2(1-4\varepsilon_0)u^2
	\min\{V(\bfv_{\boldsymbol{k}}),V(\bfv_{\boldsymbol{l}})\}.
	\]
	Mills' ratio therefore yields
	\[
	\begin{aligned}
		\Psi(z_{\boldsymbol{k},\boldsymbol{l}})
		&\le
		C\Psi(u)
		\exp\{-c\,d_{\boldsymbol{k},\boldsymbol{l}}^\alpha\}
		\exp\left\{
		-(1-4\varepsilon_0)u^2
		\min\{V(\bfv_{\boldsymbol{k}}),V(\bfv_{\boldsymbol{l}})\}
		\right\}                                                        \\
		&\le
		C\Psi(u)
		\exp\{-c\,d_{\boldsymbol{k},\boldsymbol{l}}^\alpha\}
		\left[
		\exp\{-(1-4\varepsilon_0)u^2V(\bfv_{\boldsymbol{k}})\}
		+
		\exp\{-(1-4\varepsilon_0)u^2V(\bfv_{\boldsymbol{l}})\}
		\right].
	\end{aligned}
	\]
	After replacing \(\varepsilon_0\) by a smaller number at the beginning of the
	proof, this gives
	\begin{align}
		\label{doublebound}
		\PP(A_{\boldsymbol{k}}\cap A_{\boldsymbol{l}})
		&\le
		C(1+M)^4\Psi(u)
		\exp\{-c\,d_{\boldsymbol{k},\boldsymbol{l}}^\alpha\}      \notag           \\
		&\qquad\times
		\left[
		\exp\{-(1-\varepsilon_0)u^2V(\bfv_{\boldsymbol{k}})\}
		+
		\exp\{-(1-\varepsilon_0)u^2V(\bfv_{\boldsymbol{l}})\}
		\right].
	\end{align}
	
	Moreover,
	\[
	\sup_{\boldsymbol{k}}
	\sum_{\boldsymbol{l}\ne\boldsymbol{k}}
	\exp\{-c d_{\boldsymbol{k},\boldsymbol{l}}^\alpha\}
	\le
	C\exp\{-c m^\alpha/2\}
	\]
	for all large \(M\).  Hence
	\[
	\begin{aligned}
		\sum_{\boldsymbol{k}\ne\boldsymbol{l}}
		\PP(A_{\boldsymbol{k}}\cap A_{\boldsymbol{l}})
		&\le
		\eta_2(M)\Psi(u)
		\sum_{\boldsymbol{k}}
		\exp\left\{-(1-\varepsilon_0)u^2V(\bfv_{\boldsymbol{k}})\right\},
	\end{aligned}
	\]
	where
	\[
	\eta_2(M)=C(1+M)^4\exp\{-c m^\alpha/2\}\longrightarrow0,
	\qquad M\to\infty,
	\]
	because \(m=\lfloor M^{1/2}\rfloor\).
	
	The inequality
	\[
		\exp\left\{-(1-\varepsilon_0)u^2V(\bfv_{\boldsymbol{k}})\right\}
		\leq 
		\frac{1}{M^2 q_u^2} 
		\int_{v_{\boldsymbol{k},1}-Mq_u}^{v_{\boldsymbol{k},1}}\int_{v_{\boldsymbol{k},2}-Mq_u}^{v_{\boldsymbol{k},2}}
		\exp\left\{-(1-\varepsilon_0)u^2V(x,y)\right\} \dd x \dd y
	\]
	gives a Riemann-sum upper comparison
	\[
	\sum_{\boldsymbol{k}}
	\exp\left\{-(1-\varepsilon_0)u^2V(\bfv_{\boldsymbol{k}})\right\}
	\le
	C\frac{u^{4/\alpha}}{M^2}
	\int_0^\delta\int_0^\delta
	\exp\left\{-(1-\varepsilon)u^2V(x,y)\right\}\dd x\dd y .
	\]
	Therefore
	\[
	\begin{aligned}
		\sum_{\boldsymbol{k}\ne\boldsymbol{l}}
		\PP(A_{\boldsymbol{k}}\cap A_{\boldsymbol{l}})
		&\le
		\eta(M)\Psi(u)u^{4/\alpha}
		\int_0^\delta\int_0^\delta
		\exp\left\{-(1-\varepsilon)u^2V(x,y)\right\}\dd x\dd y,
	\end{aligned}
	\]
	with \(\eta(M)\downarrow0\).
	
	Combining the lower bound for the single sum with the above estimate of the
	double sum proves \eqref{eq:lower-rectangles}.
\end{proof}

\begin{proof}[Proof of \Cref{lem:side-bounds}]
Start with the upper bound.
Consider \(S_{1,u}(\lambda)=[0,\delta]\times[0,\lambda q_u]\).  Partition only the long side into intervals of length \(Mq_u\).  The second coordinate is kept in a single block of length \(\lambda q_u\).  Applying \Cref{lem:local-pickands} gives the factor \(\calH_\alpha(M)\calH_\alpha(\lambda)\).  Since the variance loss is bounded below by \((1-o(1))x^\beta\) on a block whose long-coordinate base point is \(x\), summation over the long coordinate gives the first side contribution in \eqref{eq:upper-sides}.  The second side is identical.  The two small corner blocks contribute at most \(2(1+\varepsilon)\calH_\alpha(\lambda)^2\Psi(u)\), which is the first term in the bracket.

The lower bound is the one-dimensional version of \Cref{lem:rectangle-bounds}, applied to the restrictions \(X(t,0)\) and \(X(0,t)\).  Their local correlation is \(1-|t-s|^\alpha(1+o(1))\), and their variance loss is \(t^\beta(1+o(1))\). By \eqref{doublebound} the intersection of the two side-exceedance events has probability $O(\Psi(u))$ which is negligible compared with \(u^{2/\alpha-2/\beta}\Psi(u)\), because $R_u q_u =o(u^{-2/\beta})$ and $R_u \to \infty$ implies \(\alpha<\beta\).  This proves \eqref{eq:lower-sides}.
\end{proof}

\subsection{Proof of the integral estimates}

\begin{proof}[Proof of \Cref{lem:integrals}]
The one-dimensional estimate \eqref{eq:side-integral} follows from the change of variables \(z=\gamma^{1/\beta}u^{2/\beta}x\):
\[
        \int_0^\delta e^{-\gamma u^2x^\beta}\dd x
        \sim
        \gamma^{-1/\beta}u^{-2/\beta}\int_0^\infty e^{-z^\beta}\dd z
        =\gamma^{-1/\beta}G_\beta u^{-2/\beta}.
\]

For the two-dimensional integral $I_\gamma(u)$, set
\[
        x=u^{-2/\beta}X^{1/\beta},
        \qquad
        y=u^{-2/\beta}Y^{1/\beta},
        \qquad
        q=\frac1\beta,
        \qquad
        p=\frac a\beta .
\]
Put \(\lambda_u=u^{2-4a/\beta}\).  Then
\begin{equation}\label{eq:integral-transformed}
        I_\gamma(u)
        =\frac{u^{-4/\beta}}{\beta^2}
        \int_0^{\delta^\beta u^2}\int_0^{\delta^\beta u^2}
        X^{q-1}Y^{q-1}e^{-\gamma X-\gamma Y-\gamma\lambda_u(XY)^p}\dd X\dd Y.
\end{equation}

If \(a>\beta/2\), then \(\lambda_u\to0\), and dominated convergence in \eqref{eq:integral-transformed} gives \eqref{eq:integral-classical}.  If \(a=\beta/2\), then \(\lambda_u=1\), and the change of variables
\[
	X=x^{\beta}, \qquad Y = y^{\beta}
\]
yields \eqref{eq:integral-critical}.

It remains to consider \(0<a<\beta/2\), so that \(\lambda_u\to\infty\).  We use the following elementary asymptotic, valid for \(p,q,\gamma>0\):
\begin{equation}\label{eq:J-asymptotic}
\begin{aligned}
J(\lambda)
&=\int_0^\infty\int_0^\infty
X^{q-1}Y^{q-1}e^{-\gamma X-\gamma Y-\gamma\lambda(XY)^p}\dd X\dd Y  \\
&\sim
\frac{\Gamma(q/p)}{p^2\gamma^{q/p}}\lambda^{-q/p}\log\lambda,
        \qquad \lambda\to\infty .
\end{aligned}
\end{equation}
Indeed, make the change of variables \(Z=XY\), keeping \(X\) as the second variable.  Then
\[
J(\lambda)=\int_0^\infty Z^{q-1}e^{-\gamma\lambda Z^p}A(Z)\dd Z,
\qquad
A(Z)=\int_0^\infty X^{-1}e^{-\gamma X-\gamma Z/X}\dd X .
\]
The elementary estimate
\begin{equation}\label{eq:A-log-estimate}
        A(Z)=-\log Z+O(1),\qquad Z\downarrow0,
\end{equation}
holds because the part \(X\in[Z,1]\) gives \(-\log Z+O(1)\) (the influence of $e^{-\gamma X- \gamma Z/X}$ is negligible), while the two complementary intervals give bounded contributions.  Also \(A(Z)\) is exponentially small as \(Z\to\infty\).  Hence the integral is concentrated at \(Z=O(\lambda^{-1/p})\), and, with \(W=\lambda^{1/p}Z\),
\begin{align*}
J(\lambda)
&=\lambda^{-q/p}\int_0^\infty
        W^{q-1}e^{-\gamma W^p}
        \left(\frac{1}{p}\log\lambda-\log W+O(1)\right)\dd W  \\
&\sim
\lambda^{-q/p}\frac{\log\lambda}{p}
        \int_0^\infty W^{q-1}e^{-\gamma W^p}\dd W \\
&=\frac{\Gamma(q/p)}{p^2\gamma^{q/p}}\lambda^{-q/p}\log\lambda .
\end{align*}
This proves \eqref{eq:J-asymptotic}; the truncation at \(\delta^\beta u^2\) in \eqref{eq:integral-transformed} changes the integral only by an exponentially small quantity.

Since \(q/p=1/a\), \(p=a/\beta\), and
\[
        \log\lambda_u=\left(2-\frac{4a}{\beta}\right)\log u,
\]
combining \eqref{eq:integral-transformed} and \eqref{eq:J-asymptotic} gives
\[
        I_\gamma(u)
        \sim
        \frac{\Gamma(1/a)}{a^2\gamma^{1/a}}
        \left(2-\frac{4a}{\beta}\right)u^{-2/a}\log u,
\]
which is \eqref{eq:integral-log}.

It remains to justify the cutoff statement \eqref{eq:integral-with-cutoff}.  In the cases \(a\ge\beta/2\), the integral is concentrated on the scale \(u^{-2/\beta}\) in both coordinates (because $X$ is concentrated on the scale $1$), while \(R_uq_u=o(u^{-2/\beta})\); removing \([0,R_uq_u]\) from either coordinate is negligible.  In the logarithmic case \(a_0\le a<\beta/2\), $R_u q_u$ is of order \(u^{-2/\alpha+o(1)}\). However, $W$ is of order $1$, thus, $Z$ is of order $\lambda^{-1/p}=\lambda_u^{-\beta/a}=u^{4-2 \beta/a}$, so $X$ is of order from $u^{4-2 \beta/a}$ to $1$. Observe that for $\epsilon \in (0,1/2)$ the part $X \in [Z^{1-\epsilon},Z^\epsilon]$ in $A(Z)$ gives $-(1-2 \epsilon) \log(Z)+O(1)$, so $X$ of order from $u^{4-2\beta/a+\epsilon}$ to $u^{-\epsilon}$ give a proportion of the principal part of $I_\gamma$ tending to $1$  as $\epsilon \to 0$.

Then $Y=Z/X$ is of order from $u^{4-2 \beta/a+\epsilon}$ to $u^{-\epsilon}$. Hence, $x,y$ are of order from $u^{2/\beta-2/a+\epsilon/\beta}$ to $u^{-2/\beta-\epsilon/\beta}$.  However, given $a \geq a_0$, $\log(R_u)=o(\log(u))$ implies $R_u q_u = o(u^{-2/\alpha+\epsilon/\beta})=o(u^{2/\beta-2/a+\epsilon/\beta})$ for each $\epsilon>0$.  Thus \eqref{eq:integral-with-cutoff} follows.
\end{proof}

\begin{proof}[Proof of \Cref{lem:trend-integrals}]
The side estimate \eqref{eq:trend-side-integral} follows from the change of variables \(z=ux\):
\[
        \int_0^\delta e^{-u^2x^2-ucx}\dd x
        \sim
        u^{-1}\int_0^\infty e^{-z^2-cz}\dd z=u^{-1}L(c).
\]

For the two-dimensional integral, put \(X=ux\), \(Y=uy\).  Then
\[
        I_{\boldsymbol c}(u)
        =u^{-2}\int_0^{u\delta}\int_0^{u\delta}
        e^{-X^2-Y^2-u^{2-2a}(XY)^a-c_1X-c_2Y}\dd X\dd Y.
\]
If \(a>1\), the coefficient \(u^{2-2a}\) tends to zero, and dominated convergence gives \eqref{eq:trend-integral-classical}.  If \(a=1\), the same display directly gives \eqref{eq:trend-integral-critical}.

Let \(0<a<1\) and set \(\lambda_u=u^{2-2a}\).  For the untruncated integral in the last display, use again \(Z=XY\).  The inner integral is now
\[
A_{\boldsymbol c}(Z)
=\int_0^\infty X^{-1}
\exp\left(-X^2-(Z/X)^2-c_1X-c_2Z/X\right)\dd X .
\]
As in \eqref{eq:A-log-estimate}, \(A_{\boldsymbol c}(Z)=-\log Z+O(1)\) as \(Z\downarrow0\), and it decays exponentially as
\(Z \to \infty\).  Therefore, 
\begin{align*}
&\int_0^{u\delta}\int_0^{u\delta}
        e^{-X^2-Y^2-u^{2-2a}(XY)^a-c_1X-c_2Y}\dd X\dd Y  \\
        &\qquad\sim \int_0^\infty e^{-\lambda_u Z^a} A_{\boldsymbol{c}}(Z) \dd Z
        \\
        &\qquad= \lambda_u^{-1/a} \int_0^\infty e^{-W^a} A_{\boldsymbol{c}}(\lambda_u^{-1/a} W) \dd W
        \\
        &\qquad\sim \lambda_u^{-1/a} \int_0^\infty e^{-W^a} \left( (1/a)\log(\lambda_u)-\log(W)+O(1) \right) \dd W
        \\
&\qquad\sim
        \frac{\Gamma(1/a)}{a^2}\lambda_u^{-1/a}\log\lambda_u \\
&\qquad=
        \frac{2(1-a)\Gamma(1/a)}{a^2}u^{-2(1-a)/a}\log u .
\end{align*}
Multiplying by \(u^{-2}\) proves \eqref{eq:trend-integral-log}.  The linear trend changes only the bounded part of \(A_{\boldsymbol c}(Z)\), hence it does not change the coefficient of \(\log u\).
\end{proof}

\subsection{Proof of the main theorem}

\begin{proof}[Proof of \Cref{thm:main}]
By \ref{ass:variance} and the Borell--TIS inequality, the value of \[\PP\left\{\sup_{\bft\in[0,T]^2\backslash [0,\delta]^2}X(\bft)>u\right\}\] is exponentially smaller than \(\Psi(u)\) for all fixed \(\delta\).  Hence it is enough to consider \[\PP\left\{\sup_{\bft\in[0,\delta]^2}X(\bft)>u\right\}.\]

Recall the notation \[V(\bft)=t_1^\beta+t_2^\beta+t_1^a t_2^a.\]

First suppose \(0<a<a_0\).  By \Cref{lem:integrals},
\[
        u^{4/\alpha}
        \int_0^\delta\int_0^\delta e^{-u^2V(x,y)}\dd x\dd y
        =o\left(u^{2/\alpha-2/\beta}\right).
\]
Fix \(\lambda>0\).  On the complement of the two strips \(S_{1,u}(\lambda)\cup S_{2,u}(\lambda)\), both coordinates are at least \(\lambda q_u\).  The same union-bound argument as in \Cref{lem:rectangle-bounds}, but with no side blocks, gives an upper bound on $\PP\left\{\sup_{\bft\in[0,\delta]^2\backslash (S_{1,u}(\lambda) \cup S_{2,u}(\lambda)) }X(\bft)>u\right\}$ of order
\[
        O\left( \Psi(u)u^{4/\alpha}
        \int_0^\delta\int_0^\delta e^{-(1-o(1))u^2V(x,y)}\dd x\dd y \right)
        =o\left(u^{2/\alpha-2/\beta}\Psi(u)\right).
\]
Thus only the two thin side strips contribute to the upper bound, namely,
\[
	p(u)=\PP\left\{\sup_{\bft\in[0,T]^2}X(\bft)>u\right\}
	=\PP\left\{\sup_{\bft\in (S_{1,u}(\lambda) \cup S_{2,u}(\lambda)) }X(\bft)>u\right\}+o\left(u^{2/\alpha-2/\beta}\Psi(u)\right).
\]
The corner term in \eqref{eq:upper-sides} is also negligible because \(\alpha<\beta\) yet the integrals in \eqref{eq:upper-sides} are of order $u^{-2/\beta}$.  Therefore \Cref{lem:side-bounds}, \eqref{eq:side-integral}, and the limits
\[
        \lim_{M\to\infty}\frac{\calH_\alpha(M)}{M}=\calH_\alpha,
        \qquad
        \lim_{\lambda\downarrow0}\calH_\alpha(\lambda)=1
\]
give
\[
        \limsup_{u\to\infty}
        \frac{p(u)}{u^{2/\alpha-2/\beta}\Psi(u)}
        \le 2\calH_\alpha G_\beta .
\]
The lower bound follows from \eqref{eq:lower-sides} and \eqref{eq:side-integral}, by first sending \(M\to\infty\) and then \(\varepsilon\downarrow0\).  This proves \eqref{eq:main-side}.

Now suppose \(a_0\le a<\beta/2\).  By \Cref{lem:integrals},
\[
        I(u):=\int_0^\delta\int_0^\delta e^{-u^2V(x,y)}\dd x\dd y
        \sim
        \frac{2(\beta-2a)\Gamma(1/a)}{a^2\beta}
        u^{-2/a}\log u .
\]
The side and corner terms in \eqref{eq:upper-rectangles} are smaller than
\(u^{4/\alpha}I(u)\Psi(u)\); at \(a=a_0\), the logarithm makes the two-dimensional term larger than the side term.  Hence \eqref{eq:upper-rectangles} yields
\[
        \limsup_{u\to\infty}
        \frac{p(u)}{u^{4/\alpha-2/a}\log u\,\Psi(u)}
        \le
        \frac{\calH_\alpha(M)^2}{M^2}
        \frac{2(\beta-2a)\Gamma(1/a)}{a^2\beta(1-\varepsilon)^{1/a}}(1+\varepsilon).
\]
Letting \(M\to\infty\) and \(\varepsilon\downarrow0\) gives the desired upper bound.  The lower bound follows from \eqref{eq:lower-rectangles}, \eqref{eq:integral-with-cutoff}, and then the same limits by taking $R_u=\log(u)$, for example.  This proves \eqref{eq:main-log}.

If \(a=\beta/2\), then \Cref{lem:integrals} gives
\[
        I(u)\sim K_\beta u^{-4/\beta}.
\]
The two-dimensional term in \eqref{eq:upper-rectangles} is of order
\(u^{4/\alpha-4/\beta}\Psi(u)\), which dominates the side and corner terms because \(\alpha<\beta\).  The upper and lower bounds in \Cref{lem:rectangle-bounds} therefore give \eqref{eq:main-critical} after \(M\to\infty\) and \(\varepsilon\downarrow0\).

Finally, if \(a>\beta/2\), then
\[
        I(u)\sim G_\beta^2u^{-4/\beta},
\]
and the same argument proves \eqref{eq:main-classical}.
\end{proof}

\begin{proof}[Proof of Corollary~\ref{cor:trend}]
The proof is identical to the proof of \Cref{thm:main}, with one modification. Uniformly on a block based at \(\bfv\), the level is \(u+c_1t_1+c_2t_2=u+c_1v_1+c_2v_2+o(u^{-1})\), since $q_u=o(u^{-1})$ for $\alpha<2$, and on the relevant scales Mills' ratio gives the additional factor
\[
        \exp\{-u(c_1v_1+c_2v_2)
        +o(1)\}.
\]
Thus the variance integrals in \Cref{lem:integrals} are replaced by the trend integrals in \Cref{lem:trend-integrals}.  For \(\beta=2\), \eqref{eq:trend-side-integral} gives the side constants \(L(c_1)\) and \(L(c_2)\).  The two-dimensional estimates \eqref{eq:trend-integral-log}--\eqref{eq:trend-integral-classical} give the constants in \eqref{eq:trend-log}--\eqref{eq:trend-classical}.  The comparison of orders is the same as before, with \(a_0=2\alpha/(\alpha+2)\).  This proves all four assertions.
\end{proof}

\bibliographystyle{plainnat}
\bibliography{EEEA}

\end{document}